\documentclass[12pt, reqno, a4paper]{amsart}
\usepackage[utf8]{inputenc}
\usepackage{amsmath,amssymb,amsthm}
\usepackage{todonotes}
\usepackage[lmargin=35mm,rmargin=35mm,bmargin=30mm,tmargin=30mm]{geometry}
\usepackage{hyperref}
\usepackage{mathtools}
\newtheorem{theorem}{Theorem}

\newtheorem{lemma}[theorem]{Lemma}

\newtheorem{proposition}[theorem]{Proposition}

\newtheorem{problem}{Problem}

\newcommand{\Z}{\mathbb{Z}}
\newcommand{\R}{\mathbb{R}}
\newcommand{\N}{\mathbb{N}}

\newcommand{\eps}{\varepsilon}

\DeclareMathOperator{\ec}{ec}

\title{Exact hyperplane covers for subsets of the hypercube}
\thanks{This work has received funding from the European Research Council (ERC) under the European Union's Horizon 2020 research and innovation programme (grant agreement No~648509).}

\author[J. Aaronson]{James Aaronson}
\address{James Aaronson, Mathematical Institute, University of Oxford, Oxford OX2 6G, United Kingdom}
\email{james.aaronson.maths@gmail.com}
\author[C. Groenland]{Carla Groenland}
\address{Carla Groenland, Mathematical Institute, University of Oxford, Oxford OX2 6G, United Kingdom}
\email{groenland@maths.ox.ac.uk}
\author[A. Grzesik]{Andrzej Grzesik}
\address{Andrzej Grzesik, Faculty of Mathematics and Computer Science, Jagiellonian University, {\L}ojasiewicza 6, 30-348 Krak\'{o}w, Poland}
\email{Andrzej.Grzesik@uj.edu.pl}
\author[T. Johnston]{Tom Johnston}
\address{Tom Johnston, Mathematical Institute, University of Oxford, Oxford OX2 6G, United Kingdom}
\email{thomas.johnston@maths.ox.ac.uk}
\author[B. Kielak]{Bart{\l}omiej Kielak}
\address{Bart{\l}omiej Kielak, Faculty of Mathematics and Computer Science, Jagiellonian University, {\L}ojasiewicza 6, 30-348 Krak\'{o}w, Poland}
\email{bartlomiej.kielak@doctoral.uj.edu.pl}

\date{\today}

\begin{document}
\maketitle
\begin{abstract}
    Alon and F\"{u}redi (1993) showed that the number of hyperplanes required to cover $\{0,1\}^n\setminus \{0\}$ without covering $0$ is $n$. We initiate the study of such exact hyperplane covers of the hypercube for other subsets of the hypercube. In particular, we provide exact solutions for covering $\{0,1\}^n$ while missing up to four points and give asymptotic bounds in the general case.
    Several interesting questions are left open. 
\end{abstract}
\section{Introduction}
A vector $v\in \R^n$ and a scalar $\alpha\in \R$ determine the hyperplane 
\[
\{x\in \R^n:\langle v,x\rangle \coloneqq v_1x_1+\dots+v_nx_n=\alpha\}
\]
in $\R^n$. How many hyperplanes are needed to cover $\{0,1\}^n$? Only two are required; for instance, $\{x:x_1=0\}$ and $\{x:x_1=1\}$ will do. What happens however if $0\in \R^n$ is not allowed on any of the hyperplanes? We can `exactly' cover  $\{0,1\}^n\setminus \{0\}$ with $n$ hyperplanes: for example, the collections $\{\{x:x_i=1\}:i\in [n]\}$ or $\{\{x:\sum_{i=1}^nx_i=j\}:j \in [n]\}$ can be used, where $[n]:=\{1,2,\ldots,n\}$.
Alon and F\"{u}redi \cite{AlonFuredi} showed that in fact $n$ hyperplanes are always necessary. 

Recently, a variation was studied by Clifton and Huang \cite{CliftonHuang}, in which they require that each point from $\{0,1\}^n\setminus \{0\}$ is covered at least $k$ times for some $k\in \N$ (while $0$ is never covered). 
Another natural generalisation is to put more than just $0$ to the set of points we wish to avoid in the cover. For $B\subseteq \{0,1\}^n$, the \emph{exact cover} of $B$ is a set of hyperplanes whose union intersects $\{0,1\}^n$ exactly in~$B$ (points from $\{0,1\}^n\setminus B$ are not covered). Let $\ec(B)$ denote the \emph{exact cover number} of $B$, i.e., the minimum size of an exact cover of $B$. We will usually write $B$ in the form $\{0,1\}^n \setminus S$ for some subset $S \subseteq \{0,1\}^n$. In particular, the result of Alon and F\"{u}redi \cite{AlonFuredi} states that $\ec(\{0,1\}^n\setminus \{0\})=n$.

We first determine what happens if we remove up to four points.
\begin{theorem}
\label{thm:uptofour}
Let $S\subseteq\{0,1\}^n$.
\begin{itemize}
    \item If $|S|\in \{2,3\}$, then $\ec(\{0,1\}^n\setminus S)=n-1$.
    \item If $|S|=4$, then $\ec(\{0,1\}^n\setminus S)=n-1$ if there is a hyperplane $Q$ with $|Q\cap S|=3$ and $\ec(\{0,1\}^n\setminus S)=n-2$ otherwise.
\end{itemize}
\end{theorem}
The upper bounds are shown by iteratively reducing the dimension of the problem by one using a single `merge coordinates' hyperplane; this allows us to reduce the question to the case $n\leq 7$, which we can handle exhaustively.

Since the number of required hyperplanes seems to decrease, a natural question is whether this pattern continues. For $n\in \N$ and $k\in [2^n]$, we also introduce the exact cover numbers
\begin{align*}
\ec(n,k)&=\max\{\ec(\{0,1\}^n\setminus S):S\subseteq\{0,1\}^n,~ |S|=k\},\\
\ec(n)&=\max\{\ec(B):B\subseteq\{0,1\}^n\}.
\end{align*}
Our main result concerns the asymptotics of $\ec(n)$ and implies that $\ec(n,k)$ can be much larger than $n$.
\begin{theorem}
\label{thm:ec:arbitrary}
For any positive integer $n$, $2^{n-2}/n^2\leq \ec(n)\leq 2^{n+1}/n$.
\end{theorem}
The lower bound uses a random construction and the upper bound uses the fact that we can efficiently cover the hypercube with Hamming spheres. 

We leave open whether $\ec(n,k)\leq n$ when $k$ is sufficiently large with respect to $n$, but can show that $\ec(n,k)$ is always at most a constant (depending on $k$) away from $n$.
\begin{theorem}
\label{thm:ec:fixed_size}
For any positive integer $k$, 
\[n-\log_2(k) \leq \ec(n,k) \leq n-2^k+\ec(2^k,k).\]
\end{theorem}
The proof of this theorem uses the same techniques as the proof of Theorem \ref{thm:uptofour}.
The problem of determining the asymptotics of $\ec(n)$ was also suggested by F\"{u}redi at Alon's birthday conference in 2016.

\section{Covering all but up to four points}
\label{sec:up_to_four}

In this section, we determine $\ec(\{0,1\}^n\setminus S)$ for subsets $S$ of size 2, 3 and 4. For the lower bounds, we use the following result of Alon and F\"{u}redi \cite{AlonFuredi}.
\begin{theorem}[Corollary 1 in \cite{AlonFuredi}]
\label{thm:ATcor}
If $n\geq m\geq 1$, then $m$ hyperplanes that do not cover all vertices of $\{0,1\}^n$ miss at least $2^{n-m}$ vertices.
\end{theorem}
For the upper bounds, it suffices to give an explicit construction of a collection of hyperplanes that exactly covers $\{0,1\}^n\setminus S$, for every subset $S$ of size 2, 3 or 4.
We split the proof of Theorem \ref{thm:uptofour} into two cases, the case where $|S| \in \{2, 3\}$ and the case where $|S| = 4$.
\begin{lemma}
\label{lem:cov:23}
Let $n \geq 2$ and $S\subseteq\{0,1\}^n$ with $|S|\in \{2,3\}$. Then $\ec(\{0,1\}^n\setminus S)=n-1$.
\end{lemma}
\begin{proof}
For $n=2$ the statement is true, therefore let $n \geq 3$ and $S\subseteq\{0,1\}^n$ with $|S|\in \{2,3\}$. We first prove the lower bound $\ec(\{0,1\}^n\setminus S) \geq n-1$; this follows from applying the case of $m=n-2$ in Theorem \ref{thm:ATcor}. Indeed, this shows that any $n-2$ hyperplanes that do not cover all of $\{0,1\}^n$ miss at least $4$ vertices, and hence a minimum of $n-1$ hyperplanes are required to miss $2$ or $3$ vertices. 

For the upper bound, note that we may assume by vertex transitivity that $(0,\dots,0)\in S$. Consider first the case $|S|=2$. By relabelling the indices, we may assume the second vector $u$ in $S$ satisfies $\{i\in[n]:u_i=1\}=\{1,\dots,\ell\}$ for some $\ell\in \N$. We cover $\{0,1\}^n\setminus S$ by the collection of $n-1$ hyperplanes
\[
\{\{x:x_i=1\}:i\in \{\ell+1,\dots,n\}\}\cup \left\{\left\{x:x_1+\dots+x_\ell=j\right\}:j\in [\ell-1]\right\},
\]
noting none of these hyperplanes contain an element from $S$.

Now consider the case $|S|=3$. We may assume the second and third vectors in $S$ correspond 
to the subsets $\{1,\dots,a+b\}$ and $\{1,\dots,a\}\cup\{a+b+1,\dots,a+b+c\}$ for some $a,b,c\in \Z_{\geq 0}$ with $a+b\geq 1$ and $c\geq 1$. We first add the $n-(a+b+c)$ hyperplanes of the form $\{x:x_i=1\}$ for $i\in \{a+b+c+1,\dots,n\}$.   
For $x\in S$, we have
\begin{align*}
&x_1+\dots+x_a\in \{0,a\},\\
&x_{a+1}+\dots+x_{a+b}\in \{0,b\},\\
&x_{a+b+1}+\dots+x_{a+b+c}\in \{0,c\}.
\end{align*}
If $a\geq 1$, we add the $a-1$ hyperplanes $\{x:x_1+\dots+x_a=i\}$ for $i\in [a-1]$. Analogously, we add the $b-1$ hyperplanes $\{x: x_{a+1}+\ldots+x_{a+b}=i\}$ for $i\in[b-1]$ if $b\geq 1$, and the $c-1$ hyperplanes $\{x: x_{a+b+1}+\ldots+x_{a+b+c}=i\}$ for $i\in [c-1]$. The only points of $\{0,1\}\setminus S$ that are yet to be covered satisfy the equations above and also satisfy $x_i=0$ for $i>a+b+c$. 

Suppose first that $a,b\geq 1$. In this case we have added $n-3$ hyperplanes so far. The problem has effectively been reduced to covering $\{0,1\}^3$ with three missing points $(0,0,0), (1,1,0)\text{ and } (1,0,1)$ using $2$ hyperplanes. Indeed, we may add the following two hyperplanes to our collection in order to exactly cover $\{0,1\}^n\setminus S$:
\begin{align*}
&\left\{x:\frac{x_1+\dots+x_a}a+\frac{x_{a+1}+\dots+x_{a+b}}b+\frac{x_{a+b+1}+\dots+x_{a+b+c}}c=1\right\},\\
&\left\{x:\frac{x_{a+1}+\dots+x_{a+b}}b+\frac{x_{a+b+1}+\dots+x_{a+b+c}}c=2\right\}.
\end{align*}
Suppose now that $a=0$ or $b=0$. Since $a+b\geq 1$ and $c\geq 1$, we have used $n-2$ hyperplanes so far. If $a=0$, we may add the hyperplane 
\[
\left\{x:\frac{x_{1}+\dots+x_{b}}b+\frac{x_{b+1}+\dots+x_{b+c}}c=2\right\}
\]
and, if $b=0$, we add
\[
\left\{x:-\frac{x_1+\dots+x_a}a+\frac{x_{a+1}+\dots+x_{a+c}}c=1\right\}.
\]
In either case, the resulting collection covers $\{0,1\}\setminus S$ without covering any point in $S$.
\end{proof}

For the case of four missing points, we always need at least $n-2$ hyperplanes by Theorem \ref{thm:ATcor}. For $n=3$, we may need either $1$ or $2$ hyperplanes. For example, we may exactly cover $\{0,1\}^3\setminus (\{0\}\times \{0,1\}^2)$ by the single hyperplane $\{x:x_1=1\}$, but if $S$ does not lie on a hyperplane then we need two hyperplanes. 
The set $\{0\}\times \{0,1\}^2$ has the special property that there is no hyperplane that covers three of its points without covering the fourth.
It turns out this condition is exactly what decides how many hyperplanes are required when removing four points.
\begin{lemma}
\label{lem:cov:4}
Let $S\subseteq\{0,1\}^n$ with $|S|=4$. Then $\ec(\{0,1\}^n\setminus S)=n-1$ if there is a hyperplane $Q$ with $|Q\cap S|=3$ and $\ec(\{0,1\}^n\setminus S)=n-2$ otherwise.
\end{lemma}
\begin{proof}
We know that $\ec(\{0,1\}^n\setminus S)\geq n-2$ from Theorem \ref{thm:ATcor}. If there is a hyperplane $Q$ intersecting $S$ in exactly three points, then $\ec(\{0,1\}^n\setminus S)\geq n-1$. Indeed, by vertex transitivity, we may assume that $0$ is the point of $S$ uncovered by $Q$. Any exact cover of $\{0,1\}^n\setminus S$ can be extended to an exact cover of $\{0,1\}^n\setminus \{0\}$ by adding the hyperplane $Q$ to the collection.

We prove the claimed upper bounds by induction on $n$, handling the case $n\leq 7$ by computer search. 
Again, we may assume that $0\in S$. Let $u,v,w$ denote the other three vectors in $S$. For any $i$ with $u_i=v_i=w_i=0$, we can use a hyperplane of the form $\{x:x_i=1\}$ to reduce the covering problem to one of a lower dimension. (Note that dropping the coordinate $i$ in this case does not change whether three points in $S$ can be covered without covering the fourth.) Hence we may assume by induction that no such $i$ exists. 

After possibly permuting coordinates, we assume that $u_i=v_i=w_i=1$ on the first $a$ coordinates, $u_i=v_i=1$ and $w_i=0$ on the $b$ coordinates after that, and so on, i.e., sorted by decreasing Hamming weight and lexicographically within the same weight. In other words, our four vectors take the form
\begin{equation}
\label{eq:VennForm}
\begin{pmatrix}
0\\
u\\
v\\
w\\
\end{pmatrix}=
\begin{pmatrix}
0& 0 & 0& 0& 0 & 0& 0\\
1 & 1 & 1& 0 &1 & 0 & 0 \\
1 & 1 & 0 & 1 & 0 &1 & 0\\
1 & 0 & 1 & 1 & 0 & 0 & 1\\
\end{pmatrix}
,\end{equation}
where each column may be replaced with $0$ or more columns of its type. Since $n>7$, by the pigeonhole principle one of the columns must be repeated at least twice. We will show how to handle the case for which this is the first column (i.e. $a\geq 2$); the other cases are analogous.

Our collection of hyperplanes will contain the hyperplanes
\begin{equation}
\label{eq:ec:a}
\{\{x:x_1+\dots+x_a=i\}:i\in [a-1]\}.
\end{equation}
The only points $x$ which have yet to be covered have the property that $x_i$ takes the same value in $\{0,1\}$ for all $i\in [a]$. 
We now proceed similarly to the proof of Lemma \ref{lem:cov:23}. Informally, we wish to `merge' the first $a$ coordinates and then apply the induction hypothesis. For each $s\in S$, we define $\pi(s)=(s_{a},\dots,s_n)$. Let $\pi(S)=\{\pi(s):s\in S\}$. Then $|S|=|\pi(S)|=4$.

Any hyperplane  
\[
P=\{y:v_1y_1+\dots+v_{n-a+1}y_{n-a+1}=\alpha\}
\]
in $\{0,1\}^{n-a+1}$ can be used to define a hyperplane
\[
L(P)= \left\{x:v_1\frac{x_1+\dots+x_a}a+v_2x_{a+1}+\dots +v_{n-a+1} x_n=\alpha\right\}
\]
in $\{0,1\}^n$. For all $x\in \{0,1\}^n$ with $\sum_{i=1}^a x_i\in \{0,a\}$, we find that $\pi(x)\in P$ if and only if $x\in L(P)$. This shows that if $P_1,\dots,P_M$ form an exact cover for $\{0,1\}^{n-a+1}\setminus \pi(S)$, then $L(P_1),\dots,L(P_M)$, together with the hyperplanes from $(\ref{eq:ec:a})$, form an exact cover for $\{0,1\}^n\setminus S$. This proves 
\[
\ec(\{0,1\}^n\setminus S)\leq \ec(\{0,1\}^{n-a+1}\setminus \pi(S))+a-1.
\]
Since there is a hyperplane covering three points in $S$ without covering the fourth if and only if this is the case for $\pi(S)$, we find the claimed upper bounds by induction. 

Observe that the proof reduction works also in the case $n\leq 7$ if there are at least two coordinates of the same type in (\ref{eq:VennForm}). Thus, the computer verification is needed only in the case when each column in (\ref{eq:VennForm}) appears at most once. The code used to check the small cases is attached to the arXiv submission at \url{https://arxiv.org/abs/2010.00315}. 
\end{proof}

Another natural variant on the original Alon-F\"{u}redi problem is to ask for the exact cover number of a single layer of a hypercube without one point. It turns out this can be easily solved by translating it to the original problem. 
\begin{proposition}
Let $n\in \N$ and $i\in \{0,\dots,n\}$.
Let $B$ be obtained by removing a single point from the $i$-th layer $\{x\in \{0,1\}^n:x_1+ \ldots + x_n=i\}$. Then $\ec(B)=\min\{i,n-i\}$.
\end{proposition}
\begin{proof}
We may assume that $i\leq n/2$ and that $b=(1,\dots,1,0,\dots,0)$ is the missing point. The upper bound follows by taking the hyperplanes 
\[
\{\{x:x_1+\dots+x_i=j\}:j\in \{0,\dots,i-1\}\}.
\]
For the lower bound, we claim that we may find a cube of dimension $i$ within the $i$-th layer for which $b$ plays the role of the origin. Indeed, consider the affine map   
\[
\iota:\{0,1\}^i\to \{0,1\}^n: x \mapsto (1-x_1,1-x_2,\dots,1-x_i,0,\dots,0,x_i,x_{i-1},\dots,x_1).
\]
That is, we view the point $b$ as the origin and take the directions of the form $(-1,0,\dots,0,1)$, $(0,-1,0,\dots,0,1,0)$, etcetera, as the axes of the cube. Now $(B\setminus \{b\}) \cap \iota(\{0,1\}^i) = \iota(\{0,1\}^i\setminus\{0\})$, and hence we may convert any exact cover for $B\setminus\{b\}$ to an exact cover for $\{0,1\}^i\setminus\{0\}$. The lower bound follows from the result of Alon and F\"{u}redi \cite{AlonFuredi}.
\end{proof}

\section{Asymptotics}
\label{sec:ec:asympt}
We first consider the asymptotics of $\ec(n,k)$ when $k$ is held fixed. 
For the upper bound, we prove the following  lemma.
\begin{lemma}
\label{lem:ec:k}
For all $k\in \N$ and $n\geq 2^{k-1}$, $\ec(n,k)\leq 1+\ec(n-1,k)$.
\end{lemma}
\begin{proof}
Fix $k\in \N$, $n\geq 2^{k-1}$ and a subset $S\subseteq \{0,1\}^n$ of size $|S|=k$. For any $i\in [n]$, let $S_{-i}\subseteq\{0,1\}^{n-1}$ be obtained from $S$ by deleting coordinate $i$ from each element of $S$. 
We claim that there exists an $i\in [n]$ such that $|S_{-i}|=k$ and
\begin{equation}
    \label{eq:ec:i}
    \ec(\{0,1\}^n\setminus S)\leq 1+\ec(\{0,1\}^{n-1}\setminus S_{-i}).
\end{equation}
The lemma follows immediately from this claim. 

By vertex transitivity, we may assume that $0\in S$. 
Suppose first that there exists $i\in[n]$ for which $s_i=0$ for all $s\in S$. Then $|S_{-i}|=k$. From an exact cover for $\{0,1\}^{n-1}\setminus S_{-i}$, we may obtain an exact cover for $\{x \in \{0,1\}^n\setminus S:x_i=0\}$. Combining with the hyperplane $\{x:x_i=1\}$, this gives an exact cover for $\{0,1\}^n\setminus S$. This proves (\ref{eq:ec:i}).

We henceforth assume that $0\in S$ and that the remaining $k-1$ elements of~$S$ cannot all be $0$ on the same coordinate. Hence there are at most $2^{k-1}-1$ possible values that $(s_i:s\in S)$ can take for $i\in [n]$.  Since $n\geq 2^{k-1}$, by the pigeonhole principle, there must exist coordinates $1\leq i<j\leq n$ with $s_i=s_j$ for all $s\in S$. This implies that $|S_{-i}|=|S|=k$. We now show (\ref{eq:ec:i}) is satisfied. After permuting coordinates, we may assume that $(i,j)=(1,2)$. An exact cover for $\{0,1\}^{n-1}\setminus S_{-1}$ is converted to an exact cover for $\{0,1\}^n\setminus S$ as in the proof of Lemma \ref{lem:cov:4}: any hyperplane of the form 
\[
P=\{y:v_1y_1+\dots +v_{n-1}y_{n-1}=\alpha\}
\]
is converted to
\[
L(P)=\left\{x:v_1\frac{x_1+x_2}2+v_2x_3+\dots +v_{n-1}x_n=\alpha\right\},
\]
and we add the hyperplane $\{x:x_1+x_2=1\}$ to the adjusted collection. 
\end{proof}
It is now easy to prove to that $\ec(n,k)=n+\Theta_k(1)$.
\begin{proof}[Proof of Theorem \ref{thm:ec:fixed_size}]
Let $k\in \N$. We need to prove that for all $n\geq 2^k$,
\[
n-\log_2(k)\leq \ec(n,k)\leq n-2^k+\ec(2^k,k).
\]
The upper bound is vacuous for $n=2^k$ and follows from $n-2^k$ applications of Lemma \ref{lem:ec:k} for $n>2^k$.
The lower bound follows from Theorem \ref{thm:ATcor}: if $n-\ell$ hyperplanes cover all but $k$ vertices, then $k\geq 2^\ell$, 
and hence $n-\ell\geq n-\log_2(k)$. 
(In fact, this shows $\ec(\{0,1\}^n\setminus S)\geq n-\log_2(k)$ for each subset $S\subseteq\{0,1\}^n$ of size $k$.)
\end{proof}

We now turn to the problem of comparing exact cover numbers for sets $S$ of different sizes. We use two auxiliary lemmas.

For the lower bound, we use a random argument for which we need to know the approximate number of intersection patterns of the hypercube.
An \textit{intersection pattern} of $\{0,1\}^n$ is a non-empty subset $P\subseteq \{0,1\}^n$ for which there exists a hyperplane $H$ with $H\cap \{0,1\}^n=P$. 
\begin{lemma}
\label{lem:q_n_num_int_patterns}
$\{0,1\}^n$ has at most $2^{n^2}$ possible intersection patterns. 
\end{lemma}
\begin{proof}
We will associate each intersection pattern with a unique element from $(\{0,1\}^n)^n$.
Let $P\subseteq \{0,1\}^n$ be an intersection pattern with $P= H\cap \{0,1\}^n$ for $H$ a hyperplane. Then $|P|<2^n$. 

Let $x \in P$ be such that $\sum_{i=1}^nx_i2^i$ is minimal. 
Let $\oplus$ denote coordinate-wise addition modulo $2$ and write $x\oplus P=\{x\oplus p:p\in P\}\subseteq \{0,1\}^n$.
Note that $0\in x\oplus P$ since $x\in P$, and that $x\oplus P$ is the intersection of a linear subspace of dimension $n-1$ with $\{0,1\}^n$. (The linear subspace can be obtained from $H$ by a series of reflections.)
We greedily find $0\leq k\leq n-1$ linearly independent vectors $v_1,\dots,v_k\in x\oplus P$ whose linear span intersects $\{0,1\}^n$ in $x \oplus P$. We label $P$ with the $n$-tuple $(x,v_1,\dots,v_k,0,\dots,0)$, where we added $n-1-k$ copies of the vector $0$ at the end of the tuple.
This associates each intersection pattern to a unique element from $(\{0,1\}^n)^n$.
\end{proof}
The above proof is rather crude, but in fact not far from the truth: the number of possible intersection patterns is $2^{(1+o(1))n^2}$ (see e.g. \cite[Lemma 4.3]{Baldi}).

We also use an auxiliary result for the upper bound. The \textit{total domination number} of a graph $G$ is the minimum cardinality of a subset $D\subseteq V(G)$ such that each $v\in V(G)$ has a neighbour in $D$. \begin{lemma}[Theorem 5.2 in \cite{totaldomhypercube}]
\label{lem:total_dom_num}
The total domination number of the hypercube is at most $2^{n+1}/n$ for any $n \geq 1$. 
\end{lemma}
Note that this bound must be close to tight since the hypercube is a regular graph of degree $n$, so any total dominating set has cardinality at least $2^n/n$.

We are now ready to prove $2^{n-2}/n^2\leq \ec(n) \leq 2^{n+1}/n$. 
\begin{proof}[Proof of Theorem \ref{thm:ec:arbitrary}]
For the lower bound, we need to give a subset $B\subseteq \{0,1\}^n$ that is difficult to cover exactly. We will find a subset $S$ for which all ``large" intersection patterns have a non-empty intersection with $S$. This means that to cover $\{0,1\}^n\setminus S$, we can only use hyperplanes with ``small" intersection patterns.
We take a subset $S\subseteq \{0,1\}^n$ at random by including each point independently with probability $1/2$. Note that the lower bound is trivial for $n \leq 8$, so we may assume that $n > 8$. 

For any fixed intersection pattern $P$, the probability that it is disjoint from our random set $S$ is  $\left(\frac12\right)^{|P|}$. By Lemma \ref{lem:q_n_num_int_patterns}, there are at most $2^{n^2}$ possible intersection patterns.
Hence, by the union bound, the probability that there is an intersection pattern which has at least $2n^2$ elements and does not intersect with $S$, is at most $2^{n^2}\left(\frac12\right)^{2n^2}$, which is smaller than $1/2$ for $n \geq 2$. With probability at least $1/2$, our random set $S$ has at most $2^{n-1}$ points. Hence, there exists a subset $S$ of size $2^{n-1}$ that `hits' all intersection patterns of size at least $2n^2$.
Any exact cover for $\{0,1\}^n \setminus S$ consists entirely of hyperplanes whose intersection pattern has size at most $2n^2$, and hence needs at least $|\{0,1\}^n\setminus S|/2n^2=2^{n-2}/n^2$ hyperplanes.

We now prove the upper bound. The Hamming distance on $\{0,1\}^n$ is given by $d(x,y)=\sum_{i=1}^n |x_i-y_i|$. A Hamming sphere of radius 1 around a point $x\in \{0,1\}^n$ is given by $S(x)=\{y\in \{0,1\}^n:d(x,y)=1\}$. We claim that any subset of a Hamming sphere is an intersection pattern. Since the cube is vertex-transitive, it suffices to prove our claim for $S(0)$. The hyperplane $\{x:\sum_{i=1}^n x_i=1\}$ intersects $\{0,1\}^n$ in $S(0)$. Intersecting that hyperplane with hyperplanes of the form $\{x:x_j=0\}$ gives a lower-dimensional affine subspace, and we can construct such a subspace which intersects $S(0)$ in any subset we desire. In order to turn the affine subspace into a hyperplane with the same intersection pattern, we may add generic directions that do not yield new points in the hypercube (e.g. consider adding $(1,\pi,0,\dots,0)$). This proves each subset of a Hamming sphere is an intersection pattern.

The hypercube has total domination number at most $2^{n+1}/n$ by Lemma~\ref{lem:total_dom_num}. Hence, we can find a subset $D$ of the cube such that each vertex has a neighbour in $D$. In particular, there are $M\leq 2^{n+1}/n$ Hamming spheres centered on the vertices in $D$ that cover the cube.
For any $B\subseteq \{0,1\}^n$, we write $B=B_1\cup \dots \cup B_M$ such that each $B_i$ is covered by at least one of the Hamming spheres. This means that each $B_i$ is a intersection pattern, and we may cover $B$ exactly using $M$ hyperplanes. This gives the desired exact cover of $B$ with at most $2^{n+1}/n$ hyperplanes.
\end{proof}

Noga Alon pointed out the following improvement on the constant of the lower bound in Theorem~\ref{thm:ec:arbitrary}. There are at most $2^{n^2}$ possible intersection patterns by Lemma~\ref{lem:q_n_num_int_patterns}, so if all possible nonempty $B\subseteq \{0,1\}^n$ can be achieved by taking a union of $x$ of them, then $2^{n^2x}\geq 2^{2^n} -1$. The left-hand side of this inequality is even and the right-hand side is odd, hence $2^{n^2x}\geq 2^{2^n}$  and so $x\geq \frac{2^{n}}{n^2}$.

\section{Conclusion}
\label{sec:ec:concl}
Based on the fact that $\ec(n,k)\leq n$ for $k=1,2,3,4$, one might hope to prove that in fact $\ec(n,k)\leq n+C$ for some constant $C>0$ (independent of~$k$). However, this is not true in general by Theorem \ref{thm:ec:arbitrary}. A~natural question is then whether this will be true for $n$ sufficiently large when $k$ is fixed. 
\begin{problem}
Is there a constant $C>0$, such that for any $k\in \N$ there exists a $n_0(k)\in \N$ such that $\ec(n,k)\leq n+C$ for all $n\geq n_0(k)$?
\end{problem}
In an earlier version of this paper \cite{AaronsonGGJK20}, we conjectured that for any $S\subseteq \{0,1\}^r$ and $n\in \N$ with $n\geq r$,
\[
\ec(\{0,1\}^{n}\setminus (S\times\{0\}^{n-r}))=\ec(\{0,1\}^r\setminus S) +n-r.
\]
This would have given a negative answer to the problem above, but the counterexample $S = \{1000, 1111,1001,1011,0110,0001,0010,0111\}$ when $n = 6$ was given by Adam Zsolt Wagner \cite{Adam}.

\smallskip

One approach to improving the lower bound in Theorem \ref{thm:ec:arbitrary} is to try to prove that, for some $\eps\in (0,1)$, the number of hyperplanes containing $n^{1+\eps}$ points is $O(2^{n^{1+\eps}})$. Unfortunately, this is false: there are $2^{(1+o(1))n^2}$ possible intersection patterns of size at least $n^2$. This can be seen by considering intersection patterns of the form $\{0,1\}^{\log_2(n^2)}\times B$ for $B\subseteq\{0,1\}^{n-\log_2(n^2)}$. (If $B$ is a non-empty intersection pattern, then $\{0,1\}^{\log_2(n^2)}\times B$ is an intersection pattern containing at least $n^2$ points.) 

On the other hand,  by taking every other layer we may intersect each `axis-aligned subcube' of the form $\{0,1\}^a\times \{x\}$, ensuring that no such intersection pattern can be used in a cover. However, there is a more general type of subcube to consider.

We say a subset $A\subseteq \{0,1\}^n$ of size $|A|=2^d$ forms a $d$-dimensional \emph{subcube} if there are vectors $u,w_1,\dots,w_d\in \R^n$ such that 
\[
A=\{u+\alpha_1w_1+\dots+\alpha_d w_d: \alpha_1,\dots,\alpha_d \in \{0,1\}\}.
\]
A solution to the following problem might help improve either the upper or lower bound of Theorem \ref{thm:ec:arbitrary}.
\begin{problem}
Fix $n,d \in \N$. What is the smallest cardinality of a subset $S\subseteq\{0,1\}^n$ for which $A\cap S\neq \emptyset$ for all $d$-dimensional subcubes $A \subseteq \{0,1\}^n$?
\end{problem}
This is of a similar flavour to a problem proposed by Alon, Krech and Szab\'{o}~\cite{AlonKrechSzabo}, who asked instead for the asymptotics of the above problem when the cubes have to be axis-aligned. A $d$-dimensional axis-aligned subcube is of the form $\{0,1\}^d\times\{x\}$ after permuting coordinates. Let $g(n,d)$ denote the minimal cardinality of a subset that hits all $d$-dimensional axis-aligned subcubes in $\{0,1\}^n$. The best-known asymptotic bounds for $g(n,d)$ are from~\cite{AlonKrechSzabo}:
\[
\frac{\log_2(d)}{2^{d+2}}\leq \lim_{n\to \infty}\frac{g(n,d)}{2^n} \leq \frac1{d+1}.
\]

Finally, we remark that we have already seen these subcubes come up in Lemma \ref{lem:cov:4} as well: the sets $S\subseteq \{0,1\}^n$ of size 4 with $\ec(\{0,1\}^n\setminus S)=n-2$ are exactly the 2-dimensional subcubes.

\subsection*{Acknowledgements} 
We thank Noga Alon for pointing out to us that the problem had also been suggested by F\"{u}redi and for providing the reference~\cite{Baldi}.

We would also like to thank Alex Scott for useful discussions and the Jagiellonian University for their hospitality in hosting us during the time this research was conducted.

\bibliographystyle{abbrv} 
\bibliography{covering}   

\end{document}